\documentclass[12pt, reqno]{amsart}
\usepackage{}
\usepackage{stmaryrd}
\usepackage{amsmath}
\usepackage{amsfonts}
\usepackage{amssymb}
\usepackage[all,cmtip]{xy}           
\usepackage{xspace}
\usepackage{bbding}
\usepackage{txfonts}
\usepackage[shortlabels]{enumitem}
\usepackage{cite}
\usepackage{longtable}
\usepackage{makecell}
\usepackage{tikz}
\usepackage{titletoc}

\usepackage{ifpdf}
\ifpdf
 \usepackage[colorlinks,final,hyperindex]{hyperref}
\else
 \usepackage[colorlinks,final,backref=page,hyperindex,hypertex]{hyperref}
\fi
\usepackage[active]{srcltx}

\makeatletter

\begin{document}
\newcommand {\emptycomment}[1]{} 

\baselineskip=15pt
\newcommand{\nc}{\newcommand}
\newcommand{\delete}[1]{}
\nc{\mfootnote}[1]{\footnote{#1}} 
\nc{\todo}[1]{\tred{To do:} #1}

\nc{\mlabel}[1]{\label{#1}}  
\nc{\mcite}[1]{\cite{#1}}  
\nc{\mref}[1]{\ref{#1}}  
\nc{\meqref}[1]{\eqref{#1}} 
\nc{\mbibitem}[1]{\bibitem{#1}} 

\delete{
\nc{\mlabel}[1]{\label{#1}  
{\hfill \hspace{1cm}{\bf{{\ }\hfill(#1)}}}}
\nc{\mcite}[1]{\cite{#1}{{\bf{{\ }(#1)}}}}  
\nc{\mref}[1]{\ref{#1}{{\bf{{\ }(#1)}}}}  
\nc{\meqref}[1]{\eqref{#1}{{\bf{{\ }(#1)}}}} 
\nc{\mbibitem}[1]{\bibitem[\bf #1]{#1}} 
}

\newcommand {\comment}[1]{{\marginpar{*}\scriptsize\textbf{Comments:} #1}}
\nc{\mrm}[1]{{\rm #1}}
\nc{\id}{\mrm{id}}  \nc{\Id}{\mrm{Id}}
\nc{\admset}{\{\pm x\}\cup (-x+K^{\times}) \cup K^{\times} x^{-1}}

\def\a{\alpha}
\def\admt{admissible to~}
\def\ad{associative D-}
\def\asi{ASI~}
\def\aybe{aYBe~}
\def\b{\beta}
\def\bd{\boxdot}
\def\bbf{\overline{f}}
\def\bF{\overline{F}}
\def\bbF{\overline{\overline{F}}}
\def\bbbf{\overline{\overline{f}}}
\def\bg{\overline{g}}
\def\bG{\overline{G}}
\def\bbG{\overline{\overline{G}}}
\def\bbg{\overline{\overline{g}}}
\def\bT{\overline{T}}
\def\bt{\overline{t}}
\def\bbT{\overline{\overline{T}}}
\def\bbt{\overline{\overline{t}}}
\def\bR{\overline{R}}
\def\br{\overline{r}}
\def\bbR{\overline{\overline{R}}}
\def\bbr{\overline{\overline{r}}}
\def\bu{\overline{u}}
\def\bU{\overline{U}}
\def\bbU{\overline{\overline{U}}}
\def\bbu{\overline{\overline{u}}}
\def\bw{\overline{w}}
\def\bW{\overline{W}}
\def\bbW{\overline{\overline{W}}}
\def\bbw{\overline{\overline{w}}}
\def\btl{\blacktriangleright}
\def\btr{\blacktriangleleft}
\def\calo{\mathcal{O}}
\def\ci{\circ}
\def\d{\delta}
\def\dd{\diamondsuit}
\def\D{\Delta}
\def\frakB{\mathfrak{B}}
\def\G{\Gamma}
\def\g{\gamma}
\def\k{\kappa}
\def\l{\lambda}
\def\ll{\mathfrak{L}}
\def\lh{\leftharpoonup}
\def\lr{\longrightarrow}
\def\N{Nijenhuis~}
\def\o{\otimes}
\def\om{\omega}
\def\opa{\cdot_{A}}
\def\opb{\cdot_{B}}
\def\p{\psi}
\def\sadm{$S$-admissible~}
\def\r{\rho}
\def\ra{\rightarrow}
\def\rh{\rightharpoonup}
\def\rr{r^{\#}}
\def\s{\sigma}
\def\st{\star}
\def\ti{\times}
\def\tl{\triangleright}
\def\tr{\triangleleft}
\def\v{\varepsilon}
\def\vp{\varphi}
\def\vth{\vartheta}

\newtheorem{thm}{Theorem}[section]
\newtheorem{lem}[thm]{Lemma}
\newtheorem{cor}[thm]{Corollary}
\newtheorem{pro}[thm]{Proposition}
\theoremstyle{definition}
\newtheorem{defi}[thm]{Definition}
\newtheorem{ex}[thm]{Example}
\newtheorem{rmk}[thm]{Remark}
\newtheorem{pdef}[thm]{Proposition-Definition}
\newtheorem{condition}[thm]{Condition}
\newtheorem{question}[thm]{Question}
\renewcommand{\labelenumi}{{\rm(\alph{enumi})}}
\renewcommand{\theenumi}{\alph{enumi}}

\nc{\ts}[1]{\textcolor{purple}{MTS:#1}}
\font\cyr=wncyr10


\title{When Leibniz algebras are Nijenhuis?}

 \author{Haiying Li}
 \address{School of Mathematics and Information Science, Henan Normal University, Xinxiang 453007, China}
         \email{lihaiying@htu.edu.cn}

\author{Tianshui Ma\textsuperscript{*}}
\address{School of Mathematics and Information Science, Henan Normal University, Xinxiang 453007, China}
         \email{matianshui@htu.edu.cn}

 \author{Shuanhong Wang}
 \address{Shing-Tung Yau Center, School of Mathematics, Southeast University, Nanjing 210096, China}
          \email{shuanhwang@seu.edu.cn}

   \thanks{\textsuperscript{*}Corresponding author}

\date{\today}

 \begin{abstract} 
 Leibniz algebras can be seen as a ``non-commutative" analogue of Lie algebras. Nijenhuis operators on Leibniz algebras introduced by Cari\~{n}ena, Grabowski, and Marmo in [J. Phys. A: Math. Gen. 37(2004)] are (1, 1)-tensors with vanishing Nijenhuis torsion. Recently triangular Leibniz bialgebras were introduced by Tang and Sheng in [J. Noncommut. Geom. 16(2022)] via the twisting theory of twilled Leibniz algebras. In this paper we find that Leibniz algebras are very closely related to Nijenhuis operators, and prove that a triangular symplectic Leibniz bialgebra together with a dual triangular structure must possess Nijenhuis operators, which makes it possible to study the applications of Nijehhuis operators from the perspective of Leibniz algebras. At the same time, we regain the classical Leibniz Yang-Baxter equation by using the tensor form of classical $r$-matrics. At last we give the classification of triangular Leibniz bialgebras of low dimensions.
 \end{abstract}

\subjclass[2020]{
16T10,   
16T25,   
17B38,  
16W99,  
17B62.   
}

\keywords{Nijenhuis operators; triangular Leibniz bialgebras; symplectic Leibniz algebras}

\maketitle


 \tableofcontents


\allowdisplaybreaks

\section{Introduction and preliminaries}
 A (left) {\bf Leibniz algebra} is a pair $(\ll, [,])$ consisting of a vector space $\ll$ and a bilinear map $[,]: \ll\o \ll\lr \ll$ (write $[,](x\o y)=[x, y]$) such that for all $x, y, z\in \ll$,
 \begin{eqnarray}\label{eq:l}
 [x, [y, z]]=[[x, y], z]+[y, [x, z]].
 \end{eqnarray}
 This notion was introduced by Bloh in \cite{B} under the name D-algebras and rediscovered under the name Leibniz algebras by Loday (see \cite{L,LP}), which can be seen as a ``non-commutative" analogue of Lie algebras.

 For a given algebra, a coalgebra structure can be obtained by a dual method. It is interesting and important to give a bialgebra structure on it by some compatibility conditions between the multiplication and the comultiplication, such as Lie bialgebras (\cite{Dr,Maj}), Hopf algebras (\cite{Maj}), infinitesimal Hopf algebras (\cite{Ag99}), associative D-bialgebras (or balanced infinitesimal bialgebras, antisymmetric infinitesimal bialgebras) (\cite{Zhe,Ag00b,Bai1}), Leibniz bialgebras (\cite{TS}), dendriform-Nijenhuis bialgebras (\cite{PZGL,Le}), etc. The bialgebraic structures for Rota-Baxter associative algebras were investigated by the first named author and Liu in \cite{ML} and by Bai, Guo and the first named author in \cite{BGM}. Recently, triangular Leibniz bialgebras were introduced by Tang and Sheng in \cite{TS} by using the twisting theory of twilled Leibniz algebras. The authors also presented the equivalence between matched pairs of Leibniz algebras, Manin triples of Leibniz algebras, and Leibniz bialgebras.

 A Leibniz algebra together with a \N operator is called a \N Leibniz algebra, which is introduced by Cari\~{n}ena, Grabowski and Marmo in \cite{CGM1}. Concretely, a {\bf \N Leibniz algebra} is a triple $(\ll, [,], N)$ consisting of a Leibniz algebra $(\ll, [,])$ and a \N operator $N$, i.e., a linear map $N:\ll\lr \ll$ satisfies the equality below
 \begin{eqnarray*}
 [N(x), N(y)]+N^2([x, y])=N([N(x), y])+N([x, N(y)]),\quad \forall~x, y\in \ll.
 \end{eqnarray*}
 Nijenhuis operators are very useful in the deformation theory of algebras, bialgebras theory, integrable systems, almost complex structures, Nijenhuis geometry, preHamiltonian and Hamiltonian operators, etc (see \cite{BKM22,BKM21,CGM,CGM1,CMW,CMW1,Kon,LBS,LG,LMMP7,MLo,NN,PZGL,SX1}).

 In this paper, we find a new phenomenon that Leibniz algebras are closely related to Nijenhuis operators, and prove that a triangular symplectic Leibniz bialgebra together with a dual triangular structure must be a Nijenhuis Leibniz algebra (see Theorem \ref{thm:ln}). To this end, we will introduce the notions of dual triangular Leibniz bialgebras (Definition \ref{de:cqt}) and symplectic Leibniz algebras (Definition \ref{de:sym}). We prove that a symplectic Leibniz algebra can be obtained by a dual triangular Leibniz bialgebra (see Proposition \ref{pro:de:sym}). The dual case is also presented in Theorem \ref{thm:cln}. In the meantime, we reconsider the triangular Leibniz bialgebras via the tensor form of a classical $r$-matrix to the context of Leibniz algebras and provide the necessary and sufficient conditions for $(\ll, [,], \d_r)$ to be a Leibniz bialgebra, where the coproduct $\d_r$ is defined by Eq.(\ref{eq:cop}) (see Theorem \ref{thm:tri}). At last, we give the classification of the triangular Leibniz biaglebras of dimensions 2 and 3 (see Theorems \ref{thm:d2} and \ref{thm:d3}).

 \noindent{\bf Notations:} Throughout this paper, we fix a field $K$. All vector spaces, tensor products, and linear homomorphisms are over $K$. We denote by $\id_M$ the identity map from $M$ to $M$, $\s: M\o M\lr M\o M$ by the flip map. For a vector space $A$, $r=r^1\o r^2\in A\o A$ and $r^\s=\s(r)=r^2\o r^1$.

\section{Triangular Leibniz bialgebras revisited}\label{se:tri} In this section, we reobtain the classical Leibniz Yang-Baxter equation in \cite{TS} by using Bai's method in \cite{Bai1}, i.e., the tensor form of classical $r$-matrics.

 Dual to the notion of Leibniz algebras, we have

 \begin{defi}\label{de:cl} A {\bf Leibniz coalgebra} is a pair $(\ll, \d)$ consisting of a vector space $\ll$ and a linear map $\d: \ll\lr \ll\o \ll$ (we use Sweedler notation \cite{Sw}, write $\d(x)=x_{(1)}\o x_{(2)}$) such that for all $x\in\ll$,
 \begin{eqnarray}\label{eq:cl}
 x_{(1)}\o x_{(2)(1)}\o x_{(2)(2)}=x_{(1)(1)}\o x_{(1)(2)}\o x_{(2)}+x_{(2)(1)}\o x_{(1)}\o x_{(2)(2)}.
 \end{eqnarray}
 \end{defi}

 \begin{pro}\label{pro:de:cl}\hspace{-6mm}
 \begin{enumerate}[(1)]
   \item \label{it:de:cl1} If $(\ll, \d)$ is a Leibniz coalgebra, then $(\ll^*, \d^*)$ is a Leibniz algebra, where $\d^*:\ll^*\o \ll^*\lr \ll^*$ is defined by
   $$
   \langle\d^*(\xi\o \eta), x\rangle=\langle \xi, x_{(1)}\rangle \langle \eta, x_{(2)}\rangle, \quad \forall~~\xi, \eta\in \ll^*, x\in \ll.
   $$
   \item \label{it:de:cl2} If $(\ll, [,])$ is a finite-dimensional Leibniz algebra, then $(\ll^*, [,]^*)$ is a Leibniz coalgebra, where $[,]^*:\ll^*\lr \ll^*\o \ll^*$ is defined by
   $$
   \langle[,]^*(\xi), x\o y\rangle=\langle \xi, [x, y]\rangle, \quad \forall~~\xi\in \ll^*, x, y\in \ll.
   $$
   \item \label{it:de:cl3} If $(\ll, \d)$ is a Leibniz coalgebra, then
   \begin{eqnarray}\label{eq:cl-1}
   x_{(1)(1)}\o x_{(1)(2)}\o x_{(2)}+x_{(1)(2)}\o x_{(1)(1)}\o x_{(2)}=0, \quad \forall ~x\in \ll.
   \end{eqnarray}
 \end{enumerate}
 \end{pro}

 \begin{proof} Straightforward. \end{proof}

 \begin{rmk}\label{rmk:de:cl} Based on Proposition \ref{pro:de:cl}, for any Leibniz coalgebra, we can get a Leibniz algebra structure on the dual space. For a Leibniz algebra, under the condition of finite-dimension, one can obtain a Leibniz coalgebra structure on the dual space. These are similar to the case of associative algebras (see \cite{Sw}).
 \end{rmk}

 \begin{defi}\label{de:lb} A {\bf Leibniz bialgebra} is a triple $(\ll, [,], \d)$ consisting of a Leibniz algebra $(\ll, [,])$ and a Leibniz coalgebra $(\ll, \d)$ such that, for all $x, y\in \ll$, the following conditions hold:
 \begin{eqnarray}
 0&=&x_{(2)}\o [x_{(1)}, y]-[y_{(1)}, x]\o y_{(2)},\label{eq:b1}\\
 \d([x, y])&=&x_{(1)}\o [x_{(2)}, y]-[y, x_{(1)}]\o x_{(2)}-[x_{(1)}, y]\o x_{(2)}+x_{(2)}\o [x_{(1)}, y]\label{eq:b2} \\
 &&-[y, x_{(2)}]\o x_{(1)}-[x_{(2)}, y]\o x_{(1)}+[x, y_{(1)}]\o y_{(2)}+y_{(1)}\o [x, y_{(2)}].\nonumber
 \end{eqnarray}
 \end{defi}

 \begin{rmk} \label{rmk:de:lb} In Definition \ref{de:lb}, we replace the condition that $(\ll^*, [,]_{\ll^*})$ is a Leibniz algebra in \cite[Definition 2.13]{TS} with the condition that $(\ll, \d)$ is a Leibniz coalgebra. According to Remark \ref{rmk:de:cl}, Definition \ref{de:lb} is more general than \cite[Definition 2.13]{TS}.
 \end{rmk}

 Let $(\ll, [,])$ be a Leibniz algebra. Define the left and right multiplication $L, R: \ll\lr End(\ll)$ by $L_x(y)=[x, y]$ and $R_x(y)=[y, x]$, respectively, for all $x, y\in \ll$.

 \begin{thm} \label{thm:tri} Let $(\ll, [,])$ be a Leibniz algebra and $r\in \ll\o \ll$. Then $(\ll, [,], \d_r)$, where $\d_r$ is defined by
 \begin{equation}\label{eq:cop}
 \d(x):=\d_r(x):=-r^{1}\o [r^{2}, x]+[r^{2}, x]\o r^{1}+[x, r^{2}]\o r^{1}, \quad \forall x\in \ll,
 \end{equation}
 is a Leibniz bialgebra if and only if the following conditions hold:
 \begin{eqnarray}
 &&\hspace{25mm}(R_x\o R_y)(r-\s(r))=0,\label{eq:b1-1}\\
 &&\hspace{12mm}(L_x\o R_y+L_y\o L_x+R_y\o L_x)(r-\s(r))=0, \label{eq:b2-1}\\
 &&(\id\o L_x\o \id+\id\o R_x\o \id)(r_{12}r^{\s}_{23}+r_{13}r^{\s}_{23}-r_{12}r^{\s}_{13}-r^{\s}_{13}r_{12}) \label{eq:cl-2}\\
 &&\hspace{22mm}-(\id\o \id\o R_x)(r_{12}r_{23}+r_{13}r_{23}-r^{\s}_{12}r_{13}-r_{13}r^{\s}_{12}) \nonumber\\
 &&-(L_x\o \id\o \id+R_x\o \id\o \id)(r_{23}r^{\s}_{13}+r^{\s}_{12}r^{\s}_{13}-r^{\s}_{23}r^{\s}_{12}-r^{\s}_{12}r^{\s}_{23})=0, \nonumber
 \end{eqnarray}
 where
 \begin{eqnarray*}
 &&r_{12}r_{23}=r^1\o [r^2, \br^1]\o \br^2,~~r_{13}r_{23}=r^1\o \br^1\o [r^2, \br^2],~~ r_{23}r_{12}=r^1\o [\br^1, r^2]\o \br^2,\\
 &&r_{23}r_{13}=r^1\o \br^1\o [\br^2, r^2],~~r_{12}r_{13}=[r^1, \br^1]\o r^2\o \br^2,~~ r_{13}r_{12}=[r^1, \br^1]\o \br^2\o r^2,
 \end{eqnarray*}
 and $\br=r$.
 \end{thm}

 \begin{proof} Based on Definition \ref{de:lb}, we give the proof by the following three steps. By Eq.(\ref{eq:l}), we have
 \begin{eqnarray}\label{eq:l-1}
 [[x, y], z]+[[y, x], z]=0, \quad \forall~x, y, z\in\ll.
 \end{eqnarray}
 \begin{enumerate}[{\bf Step 1.}]
   \item \label{it:b1-1}
       \begin{eqnarray*}
       Eq.(\ref{eq:b1})&\stackrel{(\ref{eq:cop})}{\Leftrightarrow}&-[r^2, x]\o [r^1, y]+r^1\o [[r^2, x], y]+r^1\o [[x, r^2], y]\quad\stackrel{(\ref{eq:l-1})}{\Leftrightarrow}Eq.(\ref{eq:b1-1}).\\
       &&=-[r^1, x]\o [r^2, y]+[[r^2, y], x]\o r^1+[[y, r^2], x]\o r^1
       \end{eqnarray*}
   \item \label{it:b2-1}
   \begin{eqnarray*}
       Eq.(\ref{eq:b2})&\stackrel{(\ref{eq:cop})}{\Leftrightarrow}&-r^1\o [r^2, [x, y]]+[r^2, [x, y]]\o r^1+[[x, y], r^2]\o r^1+r^1\o [[r^2, x] ,y]\\
       &&-[r^2, x]\o [r^1, y]-[x, r^2]\o [r^1, y]-[y, r^1]\o [r^2, x]+[y, [r^2, x]]\o r^1\\
       &&+[y, [x, r^2]]\o r^1-[r^1, y]\o [r^2, x]+[[r^2, x], y]\o r^1+[[x, r^2], y]\o r^1\\
       &&+[r^2, x]\o [r^1, y]-r^1\o [[r^2, x] ,y]-r^1\o [[x, r^2] ,y]-[y, [r^2, x]]\o r^1\\
       &&+[y, r^1]\o [r^2, x]+[y, r^1]\o [x, r^2]-[[r^2, x], y]\o r^1+[r^1, y]\o [r^2, x]\\
       &&+[r^1, y]\o [x, r^2]+r^1\o [x, [r^2, y]]-[r^2, y]\o [x, r^1]-[y, r^2]\o [x, r^1]\\
       &&+[x, r^1]\o [r^2, y]-[x, [r^2, y]]\o r^1-[x, [y, r^2]]\o r^1=0\\
       &\stackrel{(\ref{eq:l})(\ref{eq:l-1})}{\Leftrightarrow}&Eq.(\ref{eq:b2-1}).
       \end{eqnarray*}
   \item \label{it:cl-1}
   \begin{eqnarray*}
       Eq.(\ref{eq:cl})&\stackrel{(\ref{eq:cop})}{\Leftrightarrow}&
       r^1\o \br^1\o [\br^2, [r^2, x]]-r^1\o [\br^2, [r^2, x]]\o \br^1-r^1\o [[r^2, x], \br^2]\o \br^1\\
       &&-[r^2, x]\o \br^1\o [\br^2, r^1]+[r^2, x]\o [\br^2, r^1]\o \br^1+[r^2, x]\o [r^1, \br^2]\o \br^1\\
       &&-[x, r^2]\o \br^1\o [\br^2, r^1]+[x, r^2]\o [\br^2, r^1]\o \br^1+[x, r^2]\o [r^1, \br^2]\o \br^1\\
       &&-r^1\o [r^2, \br^1]\o [\br^2, x]+[r^2, \br^1]\o r^1\o [\br^2, x]+[r^1, \br^2]\o \br^1\o [r^2, x]\\
       &&+r^1\o [r^2, [\br^2, x]]\o \br^1-[r^2, [\br^2, x]]\o r^1\o \br^1-[[r^2, x], \br^2]\o \br^1\o r^1\\
       &&+r^1\o [r^2, [x, \br^2]]\o \br^1-[r^2, [x, \br^2]]\o r^1\o \br^1-[[x, r^2], \br^2]\o \br^1\o r^1\\
       &&-r^1\o \br^1\o [r^2, [\br^2, x]]+[r^2, [\br^2, x]]\o \br^1\o r^1+[[r^2, x], \br^2]\o r^1\o \br^1\\
       &&+r^1\o [\br^2, x]\o [r^2, \br^1]-[r^2, \br^1]\o [\br^2, x]\o r^1-[r^1, \br^2]\o [r^2, x]\o \br^1\\
       &&+r^1\o [x, \br^2]\o [r^2, \br^1]-[r^2, \br^1]\o [x, \br^2]\o r^1-[r^1, \br^2]\o [x, r^2]\o \br^1=0\\
       &\stackrel{(\ref{eq:l})(\ref{eq:l-1})}{\Leftrightarrow}&
       -[r^2, x]\o \br^1\o [\br^2, r^1]+[r^2, x]\o [\br^2, r^1]\o \br^1+[r^2, x]\o [r^1, \br^2]\o \br^1\\
       &&-[x, r^2]\o \br^1\o [\br^2, r^1]+[x, r^2]\o [\br^2, r^1]\o \br^1+[x, r^2]\o [r^1, \br^2]\o \br^1\\
       &&-r^1\o [r^2, \br^1]\o [\br^2, x]+[r^2, \br^1]\o r^1\o [\br^2, x]+[r^1, \br^2]\o \br^1\o [r^2, x]\\
       &&+r^1\o [[r^2, \br^2], x]\o \br^1-[[r^2, \br^2], x]\o r^1\o \br^1+r^1\o [x, [r^2, \br^2]]\o \br^1\\
       &&-[x, [r^2, \br^2]]\o r^1\o \br^1-r^1\o \br^1\o [[r^2, \br^2], x]+r^1\o [\br^2, x]\o [r^2, \br^1]\\
       &&-[r^2, \br^1]\o [\br^2, x]\o r^1-[r^1, \br^2]\o [r^2, x]\o \br^1+r^1\o [x, \br^2]\o [r^2, \br^1]\\
       &&-[r^2, \br^1]\o [x, \br^2]\o r^1-[r^1, \br^2]\o [x, r^2]\o \br^1=0\\
       &\Leftrightarrow& Eq.(\ref{eq:cl-2}).
       \end{eqnarray*}
 \end{enumerate}
 These finish the proof.
 \end{proof}

 By Theorem \ref{thm:tri}, one can get
 \begin{cor}(\cite{TS}) \label{de:qt} Let $(\ll, [,])$ be a Leibniz algebra. If $r\in \ll\o \ll$ is a symmetric solution of the following classical Leibniz Yang-Baxter equation (cLYBe for short) in $(\ll, [,])$
 \begin{eqnarray}
 &r_{12}r_{23}+r_{13}r_{23}=r^\s_{12}r_{13}+r_{13}r^\s_{12},&\label{eq:cybe}
 \end{eqnarray}
 then $(\ll, [,], \d_r)$ is a Leibniz bialgebra, where $\d_r$ is defined by Eq.(\ref{eq:cop}). In this case, we call this Leibniz bialgebra {\bf triangular} and denoted by $(\ll, [,], r, \d_r)$.
 \end{cor}

 \begin{rmk}\label{rmk:thm:tri}
 \begin{enumerate}[(1)]
   \item \label{it:tri-0} Eqs.(\ref{eq:b1-1}) and (\ref{eq:b2-1}) demonstrate that the symmetry condition $r=\s(r)$ is natural in the subsequent discussion of Leibniz bialgebras.
   \item \label{it:tri-1} Let $(\ll, [,])$ be a Leibniz algebra. If $r\in \ll\o \ll$ is a symmetric solution of the following equation in $(\ll, [,])$
 \begin{eqnarray}
 &r_{12}r^{\s}_{23}+r_{13}r^{\s}_{23}=r_{12}r^{\s}_{13}+r^{\s}_{13}r_{12},&\label{eq:cybe-1}
 \end{eqnarray}
 then $(\ll, [,], \d_r)$ is a Leibniz bialgebra, where $\d_r$ is defined by Eq.(\ref{eq:cop}).
   \item \label{it:tri-2} Let $(\ll, [,])$ be a Leibniz algebra. If $r\in \ll\o \ll$ is a symmetric solution of the following equation in $(\ll, [,])$
 \begin{eqnarray}
 &r_{23}r^{\s}_{13}+r^{\s}_{12}r^{\s}_{13}=r^{\s}_{23}r^{\s}_{12}+r^{\s}_{12}r^{\s}_{23},&\label{eq:cybe-2}
 \end{eqnarray}
 then $(\ll, [,], \d_r)$ is a Leibniz bialgebra, where $\d_r$ is defined by Eq.(\ref{eq:cop}).
   \item \label{it:tri-3} By Items \ref{it:tri-1} and \ref{it:tri-2},   Eq.(\ref{eq:cybe-1}) or Eq.(\ref{eq:cybe-2}) also can be defined as a cLYBe. However, considering the consistency with \cite[Definition 4.9]{TS}, we choose Eq.(\ref{eq:cybe}) as a cLYBe for a Leibniz algebra.
   \item \label{it:tri-4} The cLYBe Eq.(\ref{eq:cybe}) in a Leibniz algebra is essentially different from the classical Yang-Baxter equation in a Lie algebra (see \cite[Chapter 8]{Maj}).
 \end{enumerate}
 \end{rmk}

 Now we give the characterization of triangular Leibniz bialgebras.

 \begin{pro}\label{pro:qt} Let $(\ll, [,])$ be a Leibniz algebra and $r\in \ll\o \ll$ symmetric. Then the quadruple $(\ll, [,], r, \d_r)$, where $\d_r$ is defined by Eq.(\ref{eq:cop}), is a triangular Leibniz bialgebra if and only if
 \begin{eqnarray}
 &(\d_r\o \id)(r)=r_{13}r_{23}&\label{eq:qt1}
 \end{eqnarray}
 or
 \begin{eqnarray}
 &(\id\o \d_r)(r)=r_{12}r_{13}&\label{eq:qt2}
 \end{eqnarray}
 holds.
 \end{pro}

 \begin{proof} By Eq.(\ref{eq:cop}), we have
 \begin{eqnarray}\label{eq:qt1-1}
 (\d_r\o \id)(r)
 &\stackrel{ }=&-r^1\o [r^2, \br^1]\o \br^2+[r^2, \br^1]\o r^1\o \br^2+[r^1, \br^2]\o \br^1\o r^2
 \end{eqnarray} 
 Then we obtain that Eq.(\ref{eq:qt1}) $\Leftrightarrow$ Eq.(\ref{eq:cybe}) by Eq.(\ref{eq:qt1-1}). Eq.(\ref{eq:qt2}) $\Leftrightarrow$  Eq.(\ref{eq:qt1}) by the symmetry of $r$.
 \end{proof}

 \begin{rmk}
 \begin{enumerate}[(1)]
   \item \label{it:pro:qt1} For quasitriangular Lie bialgebras, $(\id\o \d_r)(r)=r_{13}r_{12}=-r_{12}r_{13}$ (see \cite[Eq.(8.7)]{Maj}), with $\d_r(x)=[x, r^1]\o r^2+r^1\o [x, r^2]$ (see \cite[Eq.(8.5)]{Maj}). For triangular Lie bialgebras, $(\id\o \d_r)(r)=r_{12}r_{13}$ since $r^\s=-r$ (see \cite[Chapter 8]{Maj}).
   \item \label{it:pro:qt2} Eq.(\ref{eq:qt1}) $\Leftrightarrow$ the linear map $r^\#: \ll^*\lr \ll$ defined by $r^\#(\xi)=\langle \xi, r^1\rangle r^2$ is a Leibniz algebra homomorphism; Eq.(\ref{eq:qt2}) $\Leftrightarrow$ the linear map $r^{t\#}: \ll^*\lr \ll$ defined by $r^\#(\xi)=r^1\langle \xi, r^2\rangle$ is a Leibniz algebra homomorphism.
 \end{enumerate}
 \end{rmk}

\section{When Leibniz algebras are Nijenhuis?}\label{se:asibena} In this section, we will show that Leibniz algebras are closely related to Nijenhuis operators.

 In order to obtain the Nijenhuis algebraic structures from Leibniz algebras, we need to introduce the following concept.

 \begin{defi}\label{de:ccybe} Let $(\ll, \d)$ be a Leibniz coalgebra and $\om\in (\ll\o \ll)^*$ a bilinear form. For all $x, y, z\in \ll$, the equation
 \begin{eqnarray}\label{eq:ccybe}
 \om(x, y_{(1)})\om(y_{(2)}, z)+\om(x, z_{(1)})\om(y, z_{(2)})=\om(y, x_{(1)})\om(x_{(2)}, z)+\om(x_{(1)}, z)\om(y, x_{(2)})
 \end{eqnarray}
 is called a {\bf co-classical Leibniz Yang-Baxter equation (ccLYBe for short) in $(\ll, \d)$}.
 \end{defi}

 \begin{ex} \label{ex:de:ccybe} Let $\ll=K\{e, f\}$ be a two dimensional vector space, define a comultiplication on $\ll$ by
 $$\left\{
            \begin{array}{c}
             \d(e)=f\o e+f\o f \\
             \d(f)=0\hspace{20mm}\\
            \end{array}
            \right..$$
 Then $(\ll, \d)$ is a Leibniz coalgebra and

 \begin{enumerate}[(1)]
   \item \label{it:ex:de:ccybe1} \quad
          \begin{tabular}{r|rr}
          $\om$ & $e$  & $f$  \\
          \hline
           $e$ & $\l$  & $-\l$  \\
           $f$ & $0$  &  $0$ \\
        \end{tabular},
   \item \label{it:ex:de:ccybe2}\quad
          \begin{tabular}{r|rr}
          $\om$ & $e$  & $f$  \\
          \hline
           $e$ & $\l$  & $\g$  \\
           $f$ & $\g$  &  $0$ \\
        \end{tabular},
   \item \label{it:ex:de:ccybe3}\quad
          \begin{tabular}{r|rr}
          $\om$ & $e$  & $f$  \\
          \hline
           $e$ & $\l$  & $-\l$  \\
           $f$ & $-\l$ &  $\l$ \\
        \end{tabular}.
 \end{enumerate}
 are solutions of the ccLYBe in $(\ll, \d)$, where $\l, \g$ are parameters.
 \end{ex}

 \begin{thm}\label{de:cqt} Let $(\ll, \d)$ be a Leibniz coalgebra and $\om\in (\ll\o \ll)^*$ be a symmetric (in the sense of $\om(x, y)=\om(y, x)$) solution of the ccLYBe in $(\ll, \d)$. Define a multiplication on $\ll$ by
 \begin{eqnarray}\label{eq:p}
 [x, y]_{\om}=-\om(x, y_{(1)})y_{(2)}+\om(y, x_{(1)})x_{(2)}+x_{(1)}\om(y, x_{(2)}),\quad \forall~x, y\in \ll.
 \end{eqnarray}
 Then $(\ll, [,]_{\om}, \d)$ is a Leibniz bialgebra. We call this bialgebra a {\bf dual triangular Leibniz bialgebra} and denote it by $(\ll, \d, \om, [,]_{\om})$.
 \end{thm}

 \begin{proof} Eqs.(\ref{eq:b1}) and (\ref{eq:b2}) can be checked by Eq.(\ref{eq:cl-1}) and the symmetry of $\om$. Next we verify that $(\ll, [,]_\om)$ is a Leibniz algebra. For all $x, y, z\in \ll$, we compute
 \begin{eqnarray*}
 &&\hspace{-10mm}[x, [y, z]_{\om}]_{\om}-[[x, y]_{\om}, z]_{\om}-[y, [x, z]_{\om}]_{\om}\\
 &&\stackrel{(\ref{eq:p})}{=}\om(x, z_{(2)(1)})\om(y, z_{(1)})z_{(2)(2)}-\om(z_{(2)}, x_{(1)})\om(y, z_{(1)})x_{(2)}-\om(z_{(2)}, x_{(2)})\om(y, z_{(1)})x_{(1)}\\
 &&\hspace{8mm}-\om(x, y_{(2)(1)})\om(z, y_{(1)})y_{(2)(2)}+\om(y_{(2)}, x_{(1)})\om(z, y_{(1)})x_{(2)}+\om(y_{(2)}, x_{(2)})\om(z, y_{(1)})x_{(1)}\\
 &&\hspace{8mm}-\om(x, y_{(1)(1)})\om(z, y_{(2)})y_{(1)(2)}+\om(y_{(1)}, x_{(1)})\om(z, y_{(2)})x_{(2)}+\om(y_{(1)}, x_{(2)})\om(z, y_{(2)})x_{(1)}\\
 &&\hspace{8mm}-\om(x, y_{(1)})\om(y_{(2)}, z_{(1)})z_{(2)}+\om(x, y_{(1)})\om(z, y_{(2)(1)})y_{(2)(2)}+\om(x, y_{(1)})\om(z, y_{(2)(2)})y_{(2)(1)}\\
 &&\hspace{8mm}+\om(y, x_{(1)})\om(x_{(2)}, z_{(1)})z_{(2)}-\om(y, x_{(1)})\om(z, x_{(2)(1)})x_{(2)(2)}-\om(y, x_{(1)})\om(z, x_{(2)(2)})x_{(2)(1)}\\
 &&\hspace{8mm}+\om(y, x_{(2)})\om(x_{(1)}, z_{(1)})z_{(2)}-\om(y, x_{(2)})\om(z, x_{(1)(1)})x_{(1)(2)}-\om(y, x_{(2)})\om(z, x_{(1)(2)})x_{(1)(1)}\\
 &&\hspace{8mm}-\om(x, z_{(1)})\om(y, z_{(2)(1)})z_{(2)(2)}+\om(x, z_{(1)})\om(z_{(2)}, y_{(1)})y_{(2)}+\om(x, z_{(1)})\om(z_{(2)}, y_{(2)})y_{(1)}\\
 &&\hspace{8mm}+\om(z, x_{(1)})\om(y, x_{(2)(1)})x_{(2)(2)}-\om(z, x_{(1)})\om(x_{(2)}, y_{(1)})y_{(2)}-\om(z, x_{(1)})\om(x_{(2)}, y_{(2)})y_{(1)}\\
 &&\hspace{8mm}+\om(z, x_{(2)})\om(y, x_{(1)(1)})x_{(1)(2)}-\om(z, x_{(2)})\om(x_{(1)}, y_{(1)})y_{(2)}-\om(z, x_{(2)})\om(x_{(1)}, y_{(2)})y_{(1)}\\
 &&\stackrel{(\ref{eq:cl})}{=}-\om(x, z_{(1)(1)})\om(y, z_{(1)(2)})z_{(2)}+\om(x, y_{(2)(1)})\om(z, y_{(2)(2)})y_{(1)}+\om(x, y_{(1)(1)})\om(z, y_{(1)(2)})y_{(2)}\\
 &&\hspace{8mm}-\om(y, x_{(1)(1)})\om(z, x_{(1)(2)})x_{(2)}-\om(y, x_{(2)(1)})\om(z, x_{(2)(2)})x_{(1)}-\om(z_{(2)}, x_{(1)})\om(y, z_{(1)})x_{(2)}\\
 &&\hspace{8mm}-\om(z_{(2)}, x_{(2)})\om(y, z_{(1)})x_{(1)}+\om(y_{(2)}, x_{(1)})\om(z, y_{(1)})x_{(2)}+\om(y_{(2)}, x_{(2)})\om(z, y_{(1)})x_{(1)}\\
 &&\hspace{8mm}+\om(y_{(1)}, x_{(1)})\om(z, y_{(2)})x_{(2)}+\om(y_{(1)}, x_{(2)})\om(z, y_{(2)})x_{(1)}-\om(x, y_{(1)})\om(y_{(2)}, z_{(1)})z_{(2)}\\
 &&\hspace{8mm}+\om(y, x_{(1)})\om(x_{(2)}, z_{(1)})z_{(2)}+\om(y, x_{(2)})\om(x_{(1)}, z_{(1)})z_{(2)}+\om(x, z_{(1)})\om(z_{(2)}, y_{(1)})y_{(2)}\\
 &&\hspace{8mm}+\om(x, z_{(1)})\om(z_{(2)}, y_{(2)})y_{(1)}-\om(z, x_{(1)})\om(x_{(2)}, y_{(1)})y_{(2)}-\om(z, x_{(1)})\om(x_{(2)}, y_{(2)})y_{(1)}\\
 &&\hspace{8mm}-\om(z, x_{(2)})\om(x_{(1)}, y_{(1)})y_{(2)}-\om(z, x_{(2)})\om(x_{(1)}, y_{(2)})y_{(1)}\\
 &&\stackrel{(\ref{eq:ccybe})(\ref{eq:cl-1})}{=}0,
 \end{eqnarray*}
 finishing the proof.
 \end{proof}

 \begin{ex}\label{ex:de:cqt} The Leibniz coalgebra $(\ll, [,])$ together with the symmetric solutions $\om$ defined in Items \ref{it:ex:de:ccybe2} and \ref{it:ex:de:ccybe3} of the ccLYBe in Example \ref{ex:de:ccybe} is a dual triangular Leibniz bialgebra.
 \end{ex}

 \begin{pro}\label{pro:cqt} Let $(\ll, \d)$ be a Leibniz coalgebra and $\om\in (\ll\o \ll)^*$ a symmetric bilinear form. Then the quadruple $(\ll, \d, \om, [,]_{\om})$, where $[,]_{\om}$ is defined by Eq.(\ref{eq:p}), is a dual triangular Leibniz bialgebra if and only if
 \begin{eqnarray}\label{eq:cqt1}
 \om([x, y]_{\om}, z)=\om(x, z_{(1)})\om(y, z_{(2)})
 \end{eqnarray}
 or
 \begin{eqnarray}\label{eq:cqt2}
 \om(x, [y, z]_{\om})=\om(x_{(1)}, y)\om(x_{(2)}, z)
 \end{eqnarray}
 holds, where $x, y, z\in \ll$.
 \end{pro}

 \begin{proof} ($\Rightarrow$) For all $x, y, z\in \ll$, one calculates
 \begin{eqnarray*}
 \om([x, y]_{\om}, z)&\stackrel{(\ref{eq:p})}{=}&-\om(x, y_{(1)})\om(y_{(2)}, z)+\om(y, x_{(1)})\om(x_{(2)}, z)+\om(x_{(1)}, z)\om(y, x_{(2)})\\
 &\stackrel{(\ref{eq:ccybe})}{=}&\om(x, z_{(1)})\om(y, z_{(2)}).
 \end{eqnarray*}
 Then Eq.(\ref{eq:cqt1}) holds. By the symmetry of $\om$ and Eq.(\ref{eq:cqt1}), we obtain Eq.(\ref{eq:cqt2}).

 ($\Leftarrow$) is obvious.
 \end{proof}

 \begin{rmk} For dual triangular Lie bialgebras, $\om(x, [y, z]_{\om})=\om(x_{(1)}, z)\om(x_{(2)}, y)$ (see \cite[Eq.(8.13)]{Maj}), with $[x, y]_\om=x_{(1)}\om(x_{(2)}, y)-y_{(1)}\om(x, y_{(2)})$ (see \cite[Eq.(8.11)]{Maj}).
 \end{rmk}

 \begin{defi}\label{de:sym} Let $(\ll, [,])$ be a Leibniz algebra and $\om\in (\ll\o \ll)^*$ a symmetric bilinear form. Assume that for all $x, y, z\in \ll$,
 \begin{eqnarray}\label{eq:symp}
 \om([x, y], z)+\om([x, z], y)=\om([y, z], x)+\om([z, y], x).
 \end{eqnarray}
 Then we call $(\ll, [,])$ {\bf symplectic} and denote it by $(\ll, [,], \om)$.
 \end{defi}

 \begin{rmk}
 The name ``symplectic Leibniz algebra"  comes from the inspiration of symplectic Lie algebra.
 \end{rmk}

 We can get a symplectic Leibniz algebra from a dual triangular Leibniz bialgebra.

 \begin{pro}\label{pro:de:sym} Let $(\ll, \d)$ be a Leibniz coalgebra and $\om\in (\ll\o \ll)^*$ a symmetric bilinear form. If $(\ll, \d, \om, [,]_{\om})$ is a dual triangular Leibniz bialgebra, where $[,]_{\om}$ is defined by Eq.(\ref{eq:p}), then $(\ll, [,]_\om, \om)$ is a symplectic Leibniz algebra.
 \end{pro}

 \begin{proof} For all $x, y, z\in \ll$, we have
 \begin{eqnarray*}
 \om([x, y]_\om, z)+\om([x, z]_\om, y)
 &\stackrel{(\ref{eq:cqt1})}{=}&\om(x, z_{(1)})\om(y, z_{(2)})+\om(x, y_{(1)})\om(z, y_{(2)})\\
 &\stackrel{(\ref{eq:ccybe})}{=}&\om(x_{(1)}, z)\om(x_{(2)}, y)+\om(x_{(1)}, y)\om(x_{(2)}, z)\\
 &&\hspace{10mm} \hbox{~(by the symmetry of } \om)\\
 &\stackrel{(\ref{eq:cqt1})}{=}&\om([y, z]_\om, x)+\om([z, y]_\om, x),
 \end{eqnarray*}
 finishing the proof.
 \end{proof}

 Now we present the main result in this paper.

 \begin{thm}\label{thm:ln} Let $(\ll, [,], \om)$ be a symplectic Leibniz algebra and $r\in \ll\o \ll$. Suppose that $(\ll, [,], r, \d_r)$ is a triangular Leibniz bialgebra with the comultiplication $\d_r$ defined in Eq.(\ref{eq:cop}) such that $(\ll, \d_r, \om,$ $[,]_{\om})$ is a dual triangular Leibniz bialgebra with the multiplication $[,]_{\om}$ defined in Eq.(\ref{eq:p}). Define a linear map $N:\ll\lr \ll$ by
 \begin{eqnarray}
 &N(x)=\om(x, r^1)r^2,~~\forall~x\in \ll.&\label{eq:ys}
 \end{eqnarray}
 Then $(\ll, [,], N)$ is a \N Leibniz algebra.
 \end{thm}

 \begin{proof} By Eqs.(\ref{eq:ccybe}) and (\ref{eq:cop}), for all $x, y, z\in \ll$, one has
 \begin{eqnarray*}
 &\hspace{-15mm}0=-\om(x, r^1)\om(r^2 y, z)+\om(x, r^2 y)\om(r^1, z)
 +\om(x, y r^2)\om(r^1, z)-\om(x, r^1)\om(y, r^2 z)\nonumber\\
 &+\om(x, r^2 z)\om(y, r^1)+\om(x, z r^2)\om(y, r^1)-\om(y, x r^2)\om(r^1, z)-\om(x r^2, z)\om(y, r^1). 
 \end{eqnarray*}
 Since, further, $(\ll, [,], \om)$ is a symplectic Leibniz algebra, we obtain
 \begin{eqnarray}
 &\om(z, r^1)\om(x y, r^2)+\om(y, r^1)\om(x z, r^2)=\om(x, r^1)\om(y z, r^2)+\om(x, r^1)\om(z y, r^2).&\label{eq:cqt-sym}
 \end{eqnarray}
 In the rest, we check that the linear map $N$ defined in Eq.(\ref{eq:ys}) is a Nijenhuis operator on $(\ll, [,])$.
 \begin{eqnarray*}
 &&\hspace{-15mm}[N(x), N(y)]-N([N(x), y]+[x, N(y)]-N([x, y]))\\
 &\stackrel{(\ref{eq:ys})}{=}&\om(x, r^1)\om(y, \bar{r}^1)[r^2, \bar{r}^2]
 -\om(x, r^1)\om([r^2, y], \bar{r}^1)\bar{r}^2-\om(y, r^1)\om([x, r^2], \bar{r}^1)\bar{r}^2\\
 &&+\om([x, y], r^1)\om(r^2, \bar{r}^1)\bar{r}^2\\
 &\stackrel{(\ref{eq:cqt-sym})}{=}&\om(x, r^1)\om(y, \bar{r}^1)[r^2, \bar{r}^2]
 -\om(x, r^1)\om([r^2, y], \bar{r}^1)\bar{r}^2-\om(y, r^1)\om([x, r^2], \bar{r}^1)\bar{r}^2\\
 &&+\om(x, r^1)\om([y, \bar{r}^1], r^2)\bar{r}^2+\om(x, r^1)\om([\bar{r}^1, y], r^2)\bar{r}^2-\om(y, r^1)\om([x, \bar{r}^1], r^2)\bar{r}^2\\
 &\stackrel{(\ref{eq:symp})}{=}&\om(x, r^1)\om(y, \bar{r}^1)[r^2, \bar{r}^2]+\om(x, r^1)\om(y, [r^2, \bar{r}^1])\bar{r}^2-\om(x, [r^2, \bar{r}^1])\om(y, r^1)\bar{r}^2\\
 &&-\om(x, [\bar{r}^1, r^2])\om(y, r^1)\bar{r}^2 \hspace{10mm}(\hbox{by the symmetry of } \om)\\
 &\stackrel{(\ref{eq:cybe})}{=}&0.
 \end{eqnarray*}
 Therefore, $(\ll, [,], N)$ is a \N Leibniz algebra.
 \end{proof}

 \begin{ex}\label{ex:thm:ln} Let $(\ll, [,])$ be a $2$-dimensional Leibniz algebra with basis $\{e, f\}$, where $[,]$ is given by
       \begin{center}
        \begin{tabular}{r|rr}
          $[,]$ & $e$  & $f$  \\
          \hline
           $e$ & $0$  & $0$  \\
           $f$ & $e$  & $e$ \\
        \end{tabular}.
        \end{center}
 \begin{enumerate}[(1)]
   \item \label{it:thm:ln1} Set an element $r\in \ll\otimes \ll$ by
        \begin{equation*}
         r=\l e \o e+\g e \o f+\g f \o e.
        \end{equation*}
        Then $(\ll, [,], r, \d_r)$ is a triangular Leibniz bialgebra with the comultiplication $\d_r$:
        $$\left\{
            \begin{array}{ll}
             \d_r(e)=0\hspace{34mm}&\\
             \d_r(f)=(\l+\g) e\o e+\g e\o f&\\
            \end{array}
            \right.,$$
        where $\l, \g$ are parameters.
      \begin{enumerate}[(i)]
         \item \label{it:thm:ln1-1} Define a bilinear form $\om$ on $\ll$ by
        \begin{center}
        \begin{tabular}{r|rr}
          $\om$ & $e$  & $f$  \\
          \hline
           $e$ & $o$  & $\nu$  \\
           $f$ & $\nu$  &  $\kappa$ \\
        \end{tabular},
        \end{center}
        where $\nu, \kappa$ are parameters, then $(\ll, \d_r, \om, [,]_{\om})$ is a dual triangular Leibniz bialgebra with the multiplication $[,]_{\om}$ and $(\ll, [,], \om)$ is a symplectic Leibniz algebra. So by Theorem \ref{thm:ln}, let
        $$\left\{
            \begin{array}{ll}
             N(e)=\g \nu e\hspace{22mm}&\\
             N(f)=(\l \nu+\g \kappa)e+\g \nu f&\\
            \end{array}
            \right.,$$
        $(\ll, [,], N)$ is a Nijenhuis Leibniz algebra.

         \item \label{it:thm:ln1-2} Define a bilinear form $\om$ on $\ll$ by\vskip3mm
        \begin{center}
        \begin{tabular}{r|rr}
          $\om$ & $e \hspace{6mm}$  & $f \hspace{6mm}$  \\
          \hline
           $e$ & $\nu\hspace{6mm}$  & $\frac{\nu}{\g}(\l+\g)\hspace{2mm}$  \\
           $f$ & $\frac{\nu}{\g}(\l+\g)$  &  $\frac{\nu}{\g^2}(\l+\g)^2$ \\
        \end{tabular},
        \end{center}\vskip3mm
        where $\l, \g\neq 0, \nu\neq 0$ are parameters, then $(\ll, \d_r, \om, [,]_{\om})$ is a dual triangular Leibniz bialgebra with the multiplication $[,]_{\om}$. But in this case $(\ll, [,], \om)$ is not a symplectic Leibniz algebra since $\nu\neq 0$.
   \delete{So by Theorem \ref{thm:ln}, let
        $$\left\{
            \begin{array}{c}
             N(e)=(2\l+\g)\nu e+\g \nu f \hspace{21mm}\\
             N(f)=(\l+\g)\nu (\frac{2\l}{\g}+1)e+(\l+\g)\nu f \\
            \end{array}
            \right.,$$
        $(\ll, [,], N)$ is a Nijenhuis Leibniz algebra.}
        \end{enumerate}
   \item \label{it:thm:ln2} Set an element $r\in \ll\otimes \ll$ by
        \begin{equation*}
         r=\l e\o e-\l e \o f-\l f \o e+\l f\o f.
        \end{equation*}
        Then $(\ll, [,], r, \d_r)$ is a triangular Leibniz bialgebra with the comultiplication $\d_r$:
        $$\left\{
            \begin{array}{c}
             \d_r(e)=\l(e\o f-f\o e)\\
             \d_r(f)=\l(e\o f-f\o e)\\
            \end{array}
            \right.,$$
        where $\l$ is a parameter. Define a bilinear form $\om$ on $\ll$ by
        \begin{center}
        \begin{tabular}{r|rr}
          $\om$ & $e$  & $f$  \\
          \hline
           $e$ & $\g$  & $\nu$  \\
           $f$ & $\nu$  &  $\kappa$ \\
        \end{tabular},
        \end{center}
        where $\g, \nu, \kappa$ are parameters such that $\nu^2=\g \kappa$,
        then $(\ll, \d_r, \om, [,]_{\om})$ is a dual triangular Leibniz bialgebra with the multiplication $[,]_{\om}$. If further, $\g=0$, then $(\ll, [,], \om)$ is a symplectic Leibniz algebra. So by Theorem \ref{thm:ln}, let
        $$\left\{
            \begin{array}{ll}
             N(e)=0 \hspace{16mm}&\\
             N(f)=-\l \kappa (e-f)&\\
            \end{array}
            \right.,$$
        $(\ll, [,], N)$ is a Nijenhuis Leibniz algebra.

 \end{enumerate}

 \end{ex}

 In what follows, we consider the dual case of Theorem \ref{thm:ln}. For this, we need the following preparations.

 \begin{defi}\label{de:csym} Let $(\ll, \d)$ be a Leibniz coalgebra and $r\in \ll\o \ll$ a symmetric tensor. Assume that
 \begin{eqnarray}\label{eq:csymp1}
 {r^1}_{(1)}\o {r^1}_{(2)}\o r^2+{r^1}_{(1)}\o r^2\o {r^1}_{(2)}=r^2\o {r^1}_{(1)}\o {r^1}_{(2)}+r^2\o {r^1}_{(2)}\o {r^1}_{(1)}.
 \end{eqnarray}
 Then we call $(\ll, \d)$ {\bf cosymplectic} and denote it by $(\ll, \d, r)$.
 \end{defi}

 A cosymplectic Leibniz coalgebra can be constructed by a triangular Leibniz bialgebra.

 \begin{pro}\label{pro:de:csym} Let $(\ll, [, ])$ be a Leibniz algebra and $r\in \ll\o \ll$ symmetric. If $(\ll, [,], r, \d_r)$ is a triangular Leibniz bialgebra, where $\d_r$ is defined by Eq.(\ref{eq:cop}), then $(\ll, \d_r, r)$ is a cosymplectic Leibniz coalgebra.
 \end{pro}

 \begin{proof} We compute as follows.
 \begin{eqnarray*}
 {r^1}_{(1)}\o {r^1}_{(2)}\o r^2+{r^1}_{(1)}\o r^2\o {r^1}_{(2)}
 &\stackrel{(\ref{eq:qt1})}{=}&r^1\o \br^1\o [r^2, \br^2]+r^1\o [r^2, \br^2]\o \br^1\\
 &\stackrel{(\ref{eq:cybe})}{=}&[r^2, \br^2]\o r^1\o \br^1+[r^2, \br^2]\o \br^1\o r^1\\
 &&\hspace{10mm} \hbox{~(by the symmetry of } r)\\
 &\stackrel{(\ref{eq:qt1})}{=}&r^2\o {r^1}_{(1)}\o {r^1}_{(2)}+r^2\o {r^1}_{(2)}\o {r^1}_{(1)},
 \end{eqnarray*}
 finishing the proof.
 \end{proof}

 \begin{defi} A {\bf \N Leibniz coalgebra} is a triple $(\ll, \d, S)$ consisting of a Leibniz coalgebra $(\ll, \d)$ and a \N operator $S$, i.e., a linear map $S:\ll\lr \ll$ satisfies the equality below
 \begin{eqnarray}
 S(x_{(1)})\o S(x_{(2)})+S^2(x)_{(1)}\o S^2(x)_{(2)}=S(S(x)_{(1)})\o S(x)_{(2)}+S(x)_{(1)}\o S(S(x)_{(2)}),~~\forall~x\in \ll.
 \end{eqnarray}
 \end{defi}

 \begin{thm}\label{thm:cln} Let $(\ll, \d, r)$ be a cosymplectic Leibniz coalgebra and $\om\in (\ll\o \ll)^*$. Suppose that $(\ll, \d, \om, [,]_\om)$ is a dual triangular Leibniz bialgebra with the multiplication $[,]_\om$ defined in Eq.(\ref{eq:p}) such that $(\ll, [,]_\om, r, \d_r)$ is a triangular Leibniz bialgebra with the comultiplication $\d_r$ defined in Eq.(\ref{eq:cop}). Define a linear map $S:\ll\lr \ll$ by
 \begin{eqnarray}
 &S(x)=r^1\om(r^2, x),~~\forall~x\in \ll.&\label{eq:cys}
 \end{eqnarray}
 Then $(\ll, \d, S)$ is a \N Leibniz coalgebra.
 \end{thm}

 \begin{proof} Since $(\ll, [,]_\om, r, \d_r)$ is a triangular Leibniz bialgebra and $(\ll, \d, r)$ is a cosymplectic Leibniz coalgebra, we get
\begin{eqnarray}\label{eq:cqt-sym1}
 \hspace{-50mm}0\hspace{-2mm}
 &\stackrel{(\ref{eq:cybe})}{=}&\hspace{-2mm}r^1\o [r^2, \br^1]_\om\o \br^2+r^1\o \br^1\o [r^2, \br^2]_\om-[r^1, \br^1]_\om\o r^2\o \br^2-[r^1, \br^1]_\om\o \br^2\o r^2\nonumber\\
 \hspace{-2mm}&\stackrel{(\ref{eq:p})}{=}&\hspace{-2mm}-r^1\o {\br^1}_{(2)} \o \br^2 \om(r^2, {\br^1}_{(1)})+r^1\o {r^2}_{(2)}\o \br^2 \om(\br^1, {r^2}_{(1)})+r^1\o {r^2}_{(1)} \o \br^2 \om(\br^1, {r^2}_{(2)})\nonumber\\
 \hspace{-2mm}&&-r^1\o \br^1 \o {\br^2}_{(2)} \om(r^2, {\br^2}_{(1)})+r^1\o \br^1 \o {r^2}_{(2)} \om(\br^2, {r^2}_{(1)})+r^1\o \br^1 \o {r^2}_{(1)} \om(\br^2, {r^2}_{(2)})\nonumber\\
 \hspace{-2mm}&&-{r^1}_{(1)} \o r^2\o \br^2 \om(\br^1, {r^1}_{(2)})-{r^1}_{(1)}\o \br^2 \o r^2 \om(\br^1, {r^1}_{(2)})\nonumber\\ 
 \hspace{-2mm}&\stackrel{(\ref{eq:csymp1})}{=}&\hspace{-2mm}-r^1\o {\br^1}_{(1)} \o {\br^1}_{(2)} \om(r^2, \br^2)+{r^1}_{(1)}\o {r^1}_{(2)} \o \br^2 \om(\br^1, r^2)-r^1\o {\br^1}_{(2)} \o {\br^1}_{(1)} \om(r^2, \br^2)\\
 \hspace{-2mm}&&+{r^1}_{(1)}\o \br^1 \o {r^1}_{(2)} \om(\br^2, r^2).\nonumber
 \end{eqnarray}
 Now we check that the linear map $S$ defined in Eq.(\ref{eq:cys}) is a Nijenhuis operator on $(\ll, \d)$. For all $x\in \ll$,
 \begin{eqnarray*}
 &&\hspace{-13mm}S(x_{(1)})\o S(x_{(2)})-S(S(x)_{(1)})\o S(x)_{(2)}-S(x)_{(1)}\o S(S(x)_{(2)})+S^2(x)_{(1)}\o S^2(x)_{(2)}\\
 &\stackrel{(\ref{eq:cys})}{=}&r^1\o \bar{r}^1\om(r^2, x_{(1)})\om(\bar{r}^2, x_{(2)})
 -r^1\o {\bar{r}^1}_{(2)}\om(r^2, {\bar{r}^1}_{(1)})\om(\bar{r}^2, x)-{r^1}_{(1)}\o \bar{r}^1\om(r^2, x)\om(\bar{r}^2, {r^1}_{(2)})\\
 &&+{r^1}_{(1)}\o {r^1}_{(2)} \om(r^2, \br^1)\om(\br^2, x)\\
 &\stackrel{(\ref{eq:cqt-sym1})}{=}&r^1\o \bar{r}^1\om(r^2, x_{(1)})\om(\bar{r}^2, x_{(2)})
 -r^1\o {\bar{r}^1}_{(2)}\om(r^2, {\bar{r}^1}_{(1)})\om(\bar{r}^2, x)-{r^1}_{(1)}\o \bar{r}^1\om(r^2, x)\om(\bar{r}^2, {r^1}_{(2)})\\
 &&+r^1\o {\br^1}_{(1)} \om(r^2, \br^2)\om({\br^1}_{(2)}, x)+r^1\o {\br^1}_{(2)} \om(r^2, \br^2)\om({\br^1}_{(1)}, x)-{r^1}_{(1)}\o \br^1 \om(\br^2, r^2)\om({r^1}_{(2)}, x)\\ 
 &\stackrel{(\ref{eq:csymp1})}{=}&r^1\o \bar{r}^1\big(\om(r^2, x_{(1)})\om(\bar{r}^2, x_{(2)})
 -\om({r^2}_{(2)}, x)\om(\bar{r}^2, {r^2}_{(1)})+\om(r^2, {\br^2}_{(1)})\om({\br^2}_{(2)}, x)\\
 &&-\om(\bar{r}^2, {r^2}_{(2)})\om({r^2}_{(1)}, x)\big)\\
 &\stackrel{(\ref{eq:ccybe})}{=}&0.
 \end{eqnarray*}
 Therefore, $(\ll, \d, S)$ is a \N Leibniz coalgebra.
 \end{proof}

 \begin{ex}\label{ex:thm:cln} Let $\ll=K\{e, f\}$ be the two-dimensional Leibniz coalgebra defined in Example \ref{ex:de:ccybe}.
 \begin{enumerate}[(1)]
   \item \label{it:thm:ln1} By Example \ref{ex:de:cqt}, define a bilinear form $\om$ on $(\ll, \d)$ by
       \begin{center}
        \begin{tabular}{r|rr}
          $\om$ & $e$  & $f$  \\
          \hline
           $e$ & $\l$  & $\g$  \\
           $f$ & $\g$  &  $0$ \\
        \end{tabular}.
       \end{center}
        Then $(\ll, \d, \om, [,]_\om)$ is a dual triangular Leibniz bialgebra with the multiplication $[,]_\om$:
       \begin{center}
        \begin{tabular}{r|rr}
          $[,]_\om$ & $e\hspace{6mm}$  & $f\hspace{1mm}$  \\
          \hline
           $e$ & $(\l+\g)f$  & $\g f$  \\
           $f$ & $0\hspace{6mm}$  &  $0\hspace{1mm}$ \\
        \end{tabular}.
       \end{center}
        where $\l, \g$ are parameters.
      \begin{enumerate}[(i)]
         \item \label{it:thm:ln1-1} Set an element $r$ in $\ll\o \ll$ by
           \begin{eqnarray*}
           r=\nu e\o f+\nu f\o e+\kappa f\o f,
           \end{eqnarray*}
        where $\nu, \kappa$ are parameters, then $(\ll, [,]_\om, r, \d_r)$ is a triangular Leibniz bialgebra with the comultiplication $\d_r$ and $(\ll, \d, r)$ is a cosymplectic Leibniz coalgebra. So by Theorem \ref{thm:cln}, let
        $$\left\{
            \begin{array}{ll}
             S(e)=\g \nu e+(\l \nu+\g \kappa)f&\\
             S(f)=\g \nu f\hspace{22mm}&\\
            \end{array}
            \right.,$$
        $(\ll, \d, S)$ is a Nijenhuis Leibniz coalgebra.

         \item \label{it:thm:ln1-2}  Set an element $r$ in $\ll\o \ll$ by
           \begin{eqnarray*}
           r=\nu e\o e+\frac{\nu}{\g}(\l+\g) e\o f+\frac{\nu}{\g}(\l+\g) f\o e+\frac{\nu}{\g^2}(\l+\g)^2 f\o f,
           \end{eqnarray*}
        where $\nu\neq 0$ is a parameter, then $(\ll, [,]_\om, r, \d_r)$ is a triangular Leibniz bialgebra with the comultiplication $\d_r$. But in this case $(\ll, \d, r)$ is not a cosymplectic Leibniz coalgebra.
        \end{enumerate}
   \item \label{it:thm:ln2} By Example \ref{ex:de:cqt}, define a bilinear form $\om$ on $(\ll, \d)$ by
       \begin{center}
        \begin{tabular}{r|rr}
          $\om$ & $e$  & $f$  \\
          \hline
           $e$ & $\l$  & $-\l$  \\
           $f$ & $-\l$ &  $\l$ \\
        \end{tabular}.
       \end{center}
        Then $(\ll, \d, \om, [,]_\om)$ is a dual triangular Leibniz bialgebra with the multiplication $[,]_\om$:
       \begin{center}
        \begin{tabular}{r|rr}
          $[,]_\om$ & $e\hspace{6mm}$  & $f\hspace{6mm}$  \\
          \hline
           $e$ & $0\hspace{6mm}$  & $\l(e+f)$  \\
           $f$ & $-\l(e+f)$  &  $0\hspace{6mm}$ \\
        \end{tabular}.
       \end{center}
        where $\l$ is a parameter.

        Set an element $r$ in $\ll\o \ll$ by
           \begin{eqnarray*}
           r=\g~ f\o f,
           \end{eqnarray*}
        then $(\ll, [,]_\om, r, \d_r)$ is a triangular Leibniz bialgebra with the comultiplication $\d_r$ and $(\ll, \d, r)$ is a cosymplectic Leibniz coalgebra. So by Theorem \ref{thm:cln}, let
        $$\left\{
            \begin{array}{ll}
             S(e)=-\l \g f&\\
             S(f)=\l \g f \hspace{2mm}&\\
            \end{array}
            \right.,$$
        $(\ll, \d, S)$ is a Nijenhuis Leibniz coalgebra.
 \end{enumerate}

 \end{ex}

 \section{Classification of triangular Leibniz bialgebras of low dimensions}
 In this section, we give all the triangular Leibniz bialgebras based on the Leibniz algebras of dimensions 2 and 3 over the real field presented in \cite{AOR}. Throughout this section, $\l, \g, \nu, \kappa$ are parameters.

 \subsection{Classification of 2-dimensional triangular Leibniz bialgebras}

 \begin{thm}\label{thm:d2} There are two kinds of Leibniz algebras of dimension 2, see \cite{L93,Cu,AOR}. We now give all triangular Leibniz bialgebras on these 2-dimensional Leibniz algebras, see Tables \ref{tab:d2-1}-\ref{tab:d2-2}.

 \begin{enumerate}[(a)]
   \item For Leibniz algebra

   \begin{center}
        \begin{tabular}{r|rr}
          $[,]$ & $e$  & $f$  \\
          \hline
           $e$ & $0$  & $0$  \\
           $f$ & $0$  & $e$ \\
        \end{tabular}.
   \end{center}\vskip3mm

   \begin{center} Table \ref{tab:d2-1} \label{tab:d2-1}\\
   \begin{tabular}[t]{c|c}
         \hline\hline
        \bf{Leibniz $r$-matrices}  & \bf{Leibniz coalgebra induced by $r$-matrices}\\
        \hline
        $r=\l e \o e+\g e \o f+\g f \o e$
        &$\left\{
            \begin{array}{ll}
             \d_r(e)=0&\\
             \d_r(f)=\g e\o e&\\
            \end{array}
            \right.$\\
        \hline
    \end{tabular}
    \end{center}\vskip3mm

   \item For Leibniz algebra
   \begin{center}
        \begin{tabular}{r|rr}
          $[,]$ & $e$  & $f$  \\
          \hline
           $e$ & $0$  & $0$  \\
           $f$ & $e$  & $e$ \\
        \end{tabular}.
   \end{center}\vskip2mm

   \begin{center}Table \ref{tab:d2-2} \label{tab:d2-2}\\
   \begin{tabular}[t]{c|c}
   \hline\hline
        \bf{Leibniz $r$-matrices}  & \bf{Leibniz coalgebra induced by $r$-matrices}\\
        \hline
        $r=\l e \o e+\g e \o f+\g f \o e$
        & $\left\{
            \begin{array}{ll}
             \d_r(e)=0&\\
             \d_r(f)=(\l+\g) e\o e+\g e\o f&\\
            \end{array}
            \right.$ \\
            \hline
    $r=\l e\o e-\l e \o f-\l f \o e+\l f\o f$
    & $\left\{
            \begin{array}{ll}
             \d_r(e)=\l(e\o f-f\o e)&\\
             \d_r(f)=\l(e\o f-f\o e)&\\
            \end{array}
            \right..$\\
            \hline
    \end{tabular}
    \end{center}
\end{enumerate}

\end{thm}

 \subsection{Classification of 3-dimensional triangular Leibniz bialgebras}
 \begin{thm}\label{thm:d3} There are thirteen kinds of Leibniz algebras of dimension 3, see \cite[Chaper 3]{AOR}. We now give all triangular Leibniz bialgebras on these 3-dimensional Leibniz algebras, see Tables \ref{tab:d3-1}-\ref{tab:d3-11}.

 \begin{enumerate}[(1)]
   \item \label{tab:d3-1} For Leibniz algebra
       \begin{center}
        \begin{tabular}{r|rrr}
          $[,]$ & $e$  & $f$ & $g$ \\
          \hline
           $e$ & $0$  & $0$ & $0$ \\
           $f$ & $0$  & $e$ & $f$  \\
           $g$ & $-2e$  & $-f$ & $0$\\
        \end{tabular}.
        \end{center}\vskip1mm

   \begin{center} Table \ref{tab:d3-1}
   \begin{tabular}[t]{c|c}
         \hline\hline
        \bf{Leibniz $r$-matrices}  & \bf{Leibniz coalgebra induced by $r$-matrices}\\
        \hline
        $r=\l g \o g$
        &$\left\{
            \begin{array}{ll}
             \d_r(e)=-2\l (e\o g-g\o e)&\\
             \d_r(f)=\l g\o f& \\
             \d_r(g)=0&  \\
            \end{array}
            \right.$\\
        \hline
        $r=\l e\o e+\g e \o g+\g g \o e-g f\o f$
        &$\left\{
            \begin{array}{ll}
             \d_r(e)=0&\\
             \d_r(f)=-\g(e\o f-f\o e)&\\
             \makecell[c]{\d_r(g)=-2\l e\o e-2\g e\o g\\+\g f\o f}&\\
            \end{array}
            \right.$\\
            \hline
        $r=\l e\o f+\l f \o e$
        &$\left\{
            \begin{array}{ll}
             \d_r(e)=0&\\
             \d_r(f)=-\l e\o e &\\
             \d_r(g)=-3\l e\o f&\\
            \end{array}
            \right.$\\
            \hline
        $r=\l e\o e+\g e\o g+\g g \o e$
        &$\left\{
            \begin{array}{ll}
             \d_r(e)=0&\\
             \d_r(f)=\g e\o f &\\
             \d_r(g)=-2\l e\o e-2\g e\o g&\\
            \end{array}
            \right..$\\
            \hline
    \end{tabular}
    \end{center}\vskip1mm

   \item \label{tab:d3-13} For Leibniz algebra
       \begin{center}
        \begin{tabular}{r|rrr}
          $[,]$ & $e\quad$  & $f$ & $g$ \\
          \hline
           $e$ & $0\quad$  & $0$ & $0$ \\
           $f$ & $0\quad$  & $0$ & $0$  \\
           $g$ & $e+f$  & $0$ & $e$  \\
        \end{tabular}.
        \end{center}\vskip3mm

   \begin{center}Table \ref{tab:d3-13}
   \begin{tabular}[t]{c|c}
         \hline\hline
        \bf{Leibniz $r$-matrices}  & \bf{Leibniz coalgebra induced by $r$-matrices}\\
        \hline
        $\makecell[c]{r=\l e \o e+\g e \o f+\g f \o e+\nu f \o f}$
        &$\left\{
            \begin{array}{ll}
             \d_r(e)=0\hspace{56mm}\\
             \d_r(f)=0\hspace{56mm}\\
             \makecell[c]{\d_r(g)=\l e\o e+\g e\o f+\l f\o e\\+\g f\o f}\\
            \end{array}
            \right.$\\
        \hline
        $\makecell[c]{r=\l e \o f+\l f \o e+\g f\o f-\l f\o g\\-\l g\o f}$
        &$\left\{
            \begin{array}{ll}
             \d_r(e)=-\l e\o f+\l f\o e\hspace{18mm}\\
             \d_r(f)=0\hspace{46mm}\\
             \d_r(g)=-\l e\o f+\l f\o e+2\l f\o f\\
            \end{array}
            \right..$\\
            \hline
    \end{tabular}
    \end{center}

   \item \label{tab:d3-2} For Leibniz algebra
       \begin{center}
        \begin{tabular}{r|rrr}
          $[,]$ & $e$  & $f$ & $g$ \\
          \hline
           $e$ & $0$  & $0$ & $0$ \\
           $f$ & $0$  & $0$ & $f$  \\
           $g$ & $\a e$  & $-f$ & $0$  \\
        \end{tabular},
        where $0\neq\a\in \mathfrak{R}$ (real field).
        \end{center}\vskip1mm

   \begin{center}Table \ref{tab:d3-2}
   \begin{tabular}[t]{c|c}
         \hline\hline
        \bf{Leibniz $r$-matrices}  & \bf{Leibniz coalgebra induced by $r$-matrices}\\
        \hline
        $\makecell[c]{r=\l e \o e+\g e \o f+\g f \o e\\+\nu e \o g+\nu g \o e}$
        &$\left\{
            \begin{array}{ll}
             \d_r(e)=0&\\
             \d_r(f)=\nu e\o f&\\
             \makecell[c]{\d_r(g)=\l \a e\o e+((\a-1)\g+\nu)e\o f\\+\nu \a e\o g}&\\
            \end{array}
            \right.$\\
        \hline
        $r=\l e \o f+\l f \o e+\g g\o g$
        &$\left\{
            \begin{array}{ll}
             \d_r(e)=\g\a(e\o g-g\o e)&\\
             \d_r(f)=\g g\o f&\\
             \d_r(g)=\l(\a-1)e\o f&\\
            \end{array}
            \right.$\\
            \hline
        $\makecell[c]{r=\l e \o f+\l f \o e+\g f\o f+\nu f\o g\\ +\nu g\o f+\frac{\nu^2}{\g} g\o g}$
        &$\left\{
            \begin{array}{ll}
             \makecell[c]{\d_r(e)=\nu \a(e\o f-f\o e)+\frac{\nu^2 \a}{\g}(e\o g\\-g\o e)}&\\
             \d_r(f)=\nu f\o f+\frac{\nu^2}{\g} g\o f&\\
             \makecell[c]{\d_r(g)=\l(\a-1)e\o f-\g f\o f\\-\nu g\o f}&\\
            \end{array}
            \right.$\\
            \hline
        $r=\l e\o e+\g e \o f+\g f \o e+\nu f\o f$
        &$\left\{
            \begin{array}{ll}
             \d_r(e)=0&\\
             \d_r(f)=0 &\\
             \makecell[c]{\d_r(g)=\l\a e\o e+\g(\a-1) e\o f\\-\nu f\o f}&\\
            \end{array}
            \right..$\\
            \hline
    \end{tabular}
    \end{center}\vskip1mm

   \item \label{tab:d3-3} For Leibniz algebra
       \begin{center}
        \begin{tabular}{r|rrr}
          $[,]$ & $e$  & $f$ & $g$ \\
          \hline
           $e$ & $0$  & $0$ & $0$ \\
           $f$ & $0$  & $0$ & $f$  \\
           $g$ & $0$  & $-f$ & $e$  \\
        \end{tabular}.
        \end{center}\vskip3mm

   \begin{center}Table \ref{tab:d3-3}
   \begin{tabular}[t]{c|c}
         \hline\hline
        \bf{Leibniz $r$-matrices}  & \bf{Leibniz coalgebra induced by $r$-matrices}\\
        \hline
        $r=\l e \o e+\g e \o f+\g f \o e+\nu f \o f$
        &$\left\{
            \begin{array}{ll}
             \d_r(e)=0&\\
             \d_r(f)=0&\\
             \d_r(g)=-\g e\o f-\nu f\o f&\\
            \end{array}
            \right.$\\
        \hline
        $\makecell[c]{r=\l e\o e+\g e \o f+\g f \o e\\+\nu e\o g+\nu g\o e}$
        &$\left\{
            \begin{array}{ll}
             \d_r(e)=0&\\
             \d_r(f)=\nu e\o f &\\
             \d_r(g)=\nu e\o e-\g e\o f&\\
            \end{array}
            \right..$\\
            \hline
    \end{tabular}
    \end{center}\vskip3mm

   \item \label{tab:d3-4} For Leibniz algebra
       \begin{center}
        \begin{tabular}{r|rrr}
          $[,]$ & $e$  & $f$ & $g$ \\
          \hline
           $e$ & $0$  & $0$ & $0$ \\
           $f$ & $0$  & $e$ & $0$  \\
           $g$ & $0$  & $0$ & $e$  \\
        \end{tabular}.
        \end{center}\vskip3mm

   \begin{center} Table \ref{tab:d3-4}\\
   \begin{tabular}[t]{c|c}
         \hline\hline
        \bf{Leibniz $r$-matrices}  & \bf{Leibniz coalgebra induced by $r$-matrices}\\
        \hline
        $\makecell[c]{r=\l e \o e+\g e \o f+\g f \o e\\+\nu e \o g+\nu g\o e}$
        &$\left\{
            \begin{array}{ll}
             \d_r(e)=0&\\
             \d_r(f)=\g e\o e&\\
             \d_r(g)=\nu e\o e&\\
            \end{array}
            \right..$\\
        \hline
    \end{tabular}
    \end{center}\vskip3mm

   \item \label{tab:d3-5} For Leibniz algebra
       \begin{center}
        \begin{tabular}{r|rrr}
          $[,]$ & $e$  & $f$ & $g$ \\
          \hline
           $e$ & $0$  & $0$ & $0$ \\
           $f$ & $0$  & $e$ & $0$  \\
           $g$ & $0$  & $0$ & $-e$  \\
        \end{tabular}.
        \end{center}\vskip3mm

   \begin{center} Table \ref{tab:d3-5}
   \begin{tabular}[t]{c|c}
         \hline\hline
        \bf{Leibniz $r$-matrices}  & \bf{Leibniz coalgebra induced by $r$-matrices}\\
        \hline
        $\makecell[c]{r=\l e \o e+\g e \o f+\g f \o e+\nu e \o g\\+\nu g\o e}$
        &$\left\{
            \begin{array}{ll}
             \d_r(e)=0\hspace{10mm}\\
             \d_r(f)=\g e\o e\hspace{2mm}\\
             \d_r(g)=-\nu e\o e\\
            \end{array}
            \right.$\\
        \hline
        $\makecell[c]{r=\l e\o e+\g e \o f+\g f \o e+\g e\o g\\+\g g\o e+\nu f\o f+\nu f\o g\\+\nu g\o f+\nu g\o g}$
        &$\left\{
            \begin{array}{ll}
             \d_r(e)=0\hspace{76mm}\\
             \makecell[c]{\d_r(f)=\g e\o e+2\nu e\o f-\nu f\o e\\+2\nu e\o g-\nu g\o e} \hspace{3mm}\\
             \makecell[c]{\d_r(g)=-\g e\o e-2\nu e\o f+\nu f\o e\\-2\nu e\o g+\nu g\o e}\\
            \end{array}
            \right.$\\
            \hline
        $\makecell[c]{r=\l e\o e+\g e \o f+\g f \o e-\g e\o g\\-\g g\o e+\nu f\o f-\nu f\o g\\-\nu g\o f+\nu g\o g}$
        &$\left\{
            \begin{array}{ll}
             \d_r(e)=0\hspace{73mm}\\
             \makecell[c]{\d_r(f)=\g e\o e+2\nu e\o f-\nu f\o e\\-2\nu e\o g+\nu g\o e}\\
             \makecell[c]{\d_r(g)=\g e\o e+2\nu e\o f-\nu f\o e\\-2\nu e\o g+\nu g\o e}\\
            \end{array}
            \right..$\\
            \hline
    \end{tabular}
    \end{center}\vskip3mm

   \item \label{tab:d3-6} For Leibniz algebra
       \begin{center}
        \begin{tabular}{r|rrr}
          $[,]$ & $e$  & $f$ & $g$ \\
          \hline
           $e$ & $0$  & $0$ & $0$ \\
           $f$ & $0$  & $e$ & $0$  \\
           $g$ & $0$  & $e$ & $\a e$  \\
        \end{tabular},
        where $0\neq \a\in \mathfrak{R}$.
        \end{center}\vskip3mm

   \begin{center}Table \ref{tab:d3-6}
   \begin{tabular}[t]{c|c}
         \hline\hline
        \bf{Leibniz $r$-matrices}  & \bf{Leibniz coalgebra induced by $r$-matrices}\\
        \hline
        $\makecell[c]{r=\l e \o e+\g e \o f+\g f \o e\\+\nu e \o g+\nu g\o e}$
        &$\left\{
            \begin{array}{ll}
             \d_r(e)=0\hspace{20mm}\\
             \d_r(f)=\g e\o e\hspace{12mm}\\
             \d_r(g)=(\g+\nu \a) e\o e\\
            \end{array}
            \right.$\\
        \hline
        $\makecell[c]{r=\l e\o e+\g (e \o f+f \o e)\\+\frac{\g \kappa}{\nu} (e\o g+g\o e)\\-(\nu+\kappa \a) f\o f+\nu (f\o g\\+g\o f)+\kappa g\o g}$
        &$\left\{
            \begin{array}{ll}
             \d_r(e)=0\hspace{125mm}&\\
             \makecell[c]{\d_r(f)=\g e\o e-(\nu+2\kappa \a) e\o f+\kappa \a f\o e+(2\nu \\+\kappa) e\o g-(\nu+\kappa) g\o e \hspace{19mm}}&\\
             \makecell[c]{\d_r(g)=\frac{\g}{\nu}(\nu+\kappa \a) e\o e+(2\nu \a-\nu-\kappa \a) e\o f\\-\nu \a f\o e+(2\kappa \a+\nu) e\o g\\-\kappa \a g\o e}&\\
        where~~\nu\neq 0, 0\neq \a\leq \frac{1}{4}$ and $\nu=\frac{-1\pm \sqrt{1-4\a}}{2}\kappa.&\\
            \end{array}
            \right..$\\
            \hline
    \end{tabular}
    \end{center}\vskip3mm

   \item \label{tab:d3-7} For Leibniz algebra
       \begin{center}
        \begin{tabular}{r|rrr}
          $[,]$ & $e$  & $f$ & $g$ \\
          \hline
           $e$ & $0$  & $0$ & $0$ \\
           $f$ & $0$  & $0$ & $0$  \\
           $g$ & $0$  & $e$ & $0$  \\
        \end{tabular}.
        \end{center}\vskip3mm

   \begin{center}Table \ref{tab:d3-7}
   \begin{tabular}[t]{c|c}
         \hline\hline
        \bf{Leibniz $r$-matrices}  & \bf{Leibniz coalgebra induced by $r$-matrices}\\
        \hline
        $\makecell[c]{r=\l e \o e+\g e \o f+\g f \o e+\nu e\o g\\+\nu g\o e+\kappa f \o f}$
        &$\left\{
            \begin{array}{ll}
             \d_r(e)=0\\
             \d_r(f)=0\\
             \d_r(g)=0\\
            \end{array}
            \right.$\\
        \hline
        $r=\l e\o e+\g e \o f+\g f \o e+\nu g\o g$
        &$\left\{
            \begin{array}{ll}
             \d_r(e)=0\hspace{23mm}\\
             \d_r(f)=\nu e\o g-\nu g\o e\\
             \d_r(g)=0 \hspace{23mm}\\
            \end{array}
            \right..$\\
            \hline
    \end{tabular}
    \end{center}\vskip3mm

   \item \label{tab:d3-8} For Leibniz algebra
       \begin{center}
        \begin{tabular}{r|rrr}
          $[,]$ & $e$  & $f$ & $g$ \\
          \hline
           $e$ & $0$  & $0$ & $0$ \\
           $f$ & $0$  & $0$ & $0$  \\
           $g$ & $f$  & $e$ & $0$  \\
        \end{tabular}.
        \end{center}\vskip3mm

   \begin{center}Table \ref{tab:d3-8}
   \begin{tabular}[t]{c|c}
         \hline\hline
        \bf{Leibniz $r$-matrices}  & \bf{Leibniz coalgebra induced by $r$-matrices}\\
        \hline
        $r=\l g\o g$
        &$\left\{
            \begin{array}{ll}
             \d_r(e)=\l f\o g-\l g\o f\\
             \d_r(f)=\l e\o g-\l g\o e\\
             \d_r(g)=0\hspace{23mm}\\
            \end{array}
            \right.$\\
        \hline
        $\makecell[c]{r=\l e \o e-\l e \o f-\l f \o e+\nu e\o g\\+\nu g\o e+\l f \o f-\nu f\o g\\-\nu g\o f}$
        &$\left\{
            \begin{array}{ll}
             \d_r(e)=-\nu e\o f+\nu f\o e\hspace{62mm}\\
             \d_r(f)=-\nu e\o f+\nu f\o e\hspace{62mm}\\
             \makecell[c]{\d_r(g)=-\l e\o e+\l e\o f+\l f\o e\\\qquad-\l f\o f-\nu e\o g+\nu f\o g}\\
            \end{array}
            \right.$\\
            \hline
        $r=\l e\o g+\l g\o e+\l f \o g+\l g\o f$
        &$\left\{
            \begin{array}{ll}
             \d_r(e)=-\l e\o f+\l f\o e\\
             \d_r(f)=\l e\o f-\l f\o e\\
             \d_r(g)=\l e\o g+\l f\o g\\
            \end{array}
            \right.$\\
            \hline
        $r=\l e\o e+\g e \o f+\g f \o e+\nu f\o f$
        &$\left\{
            \begin{array}{ll}
             \d_r(e)=0\hspace{56mm}\\
             \d_r(f)=0\hspace{56mm}\\
             \makecell[c]{\d_r(g)=\g e\o e+\nu e\o f+\l f\o e\\+\g f\o f}\\
            \end{array}
            \right..$\\
            \hline
    \end{tabular}
    \end{center}\vskip3mm

   \item \label{tab:d3-9} For Leibniz algebra
       \begin{center}
        \begin{tabular}{r|rrr}
          $[,]$ & $e$  & $f$ & $g$ \\
          \hline
           $e$ & $0$  & $0$ & $0$ \\
           $f$ & $0$  & $0$ & $0$  \\
           $g$ & $f$  & $-e$ & $0$  \\
        \end{tabular}.
        \end{center}\vskip3mm

   \begin{center}Table \ref{tab:d3-9}
   \begin{tabular}[t]{c|c}
         \hline\hline
        \bf{Leibniz $r$-matrices}  & \bf{Leibniz coalgebra induced by $r$-matrices}\\
        \hline
        $r=\l g\o g$
        &$\left\{
            \begin{array}{ll}
             \d_r(e)=\l f\o g-\l g\o f\\
             \d_r(f)=-\l e\o g+\l g\o e\\
             \d_r(g)=0\hspace{26mm}\\
            \end{array}
            \right.$\\
        \hline
        $r=\l e \o e+\g e \o f+\g f \o e+\nu f \o f$
        &$\left\{
            \begin{array}{ll}
             \d_r(e)=0\hspace{58mm}\\
             \d_r(f)=0\hspace{58mm}\\
             \makecell[c]{\d_r(g)=-\g e\o e-\nu e\o f\\\qquad\quad+\l f\o e+\g f\o f}\\
            \end{array}
            \right..$\\
            \hline
    \end{tabular}
    \end{center}\vskip3mm

   \item \label{tab:d3-10} For Leibniz algebra
       \begin{center}
        \begin{tabular}{r|rrr}
          $[,]$ & $e$  & $f\quad$ & $g$ \\
          \hline
           $e$ & $0$  & $0\quad$ & $0$ \\
           $f$ & $0$  & $0\quad$ & $0$  \\
           $g$ & $f$  & $\a e+f$ & $0$  \\
        \end{tabular},
        where $0\neq \a\in \mathfrak{R}$.
        \end{center}\vskip3mm

   \begin{center}Table \ref{tab:d3-10}
   \begin{tabular}[t]{c|c}
         \hline\hline
        \bf{Leibniz $r$-matrices}  & \bf{Leibniz coalgebra induced by $r$-matrices}\\
        \hline
        $r=\l g\o g$
        &$\left\{
            \begin{array}{ll}
             \d_r(e)=\l f\o g-\l g\o f\hspace{36mm}\\
             \makecell[c]{\d_r(f)=\l \a e\o g-\l \a g\o e\\\qquad\quad+\l f\o g-\l g\o f}\\
             \d_r(g)=0\hspace{60mm}\\
            \end{array}
            \right.$\\
        \hline
        $\makecell[c]{r=\l e \o e+\g e \o f+\g f \o e+\nu f \o f}$
        &$\left\{
            \begin{array}{ll}
             \d_r(e)=0\hspace{80mm}\\
             \d_r(f)=0\hspace{80mm}\\
             \makecell[c]{\d_r(g)=\g \a e\o e+\nu \a e\o f+(\l+\g) \\\times f\o e+(\g+\nu) f\o f}\\
            \end{array}
            \right.$\\
            \hline
        $\makecell[c]{r=\g e\o g+\g g\o e+\l f\o g+\l g\o f,\\~\hbox{where~}\g=\frac{-1\pm \sqrt{5}}{2}\l \hbox{~and~} \a=1}$
        &$\left\{
            \begin{array}{ll}
             \d_r(e)=-\g e\o f+\g f\o e\hspace{16mm}\\
             \d_r(f)=(\l-\g) e\o f+(\g-\l) f\o e\\
             \d_r(g)=\l e\o g+(\l+\g) f\o g\hspace{9mm}\\
            \end{array}
            \right.$\\
            \hline
        $\makecell[c]{r=(\frac{\g \nu}{\l^2}(\nu-2\l)+\g) e \o e+(\frac{\g \nu}{\l}-\g) (e \o f\\+f \o e)+\l (e\o g+g\o e)\\+\g f \o f+\nu (f\o g+g\o f),\\~\hbox{where~}\nu=\frac{1\pm \sqrt{5}}{2}\l,~\l\neq 0 \hbox{~and~} \a=1}$
        &$\left\{
            \begin{array}{ll}
             \d_r(e)=-\l e\o f+\l f\o e\hspace{83mm}\\
             \d_r(f)=(\nu-\l) e\o f+(\l-\nu) f\o e\hspace{68mm}\\
             \makecell[c]{\d_r(g)=\frac{\nu-\l}{\l}\g e\o e+\g e\o f+\frac{(\nu-\l)\g \nu}{\l}\\\qquad\quad\times f\o e+\frac{\g \nu}{\l} f\o f+\nu e\o g\\+(\l+\nu) f\o g}\\
            \end{array}
            \right..$\\
            \hline
    \end{tabular}
    \end{center}\vskip3mm

   \item \label{tab:d3-12} For Leibniz algebra
       \begin{center}
        \begin{tabular}{r|rrr}
          $[,]$ & $e$  & $f$ & $g$ \\
          \hline
           $e$ & $0$  & $0$ & $0$ \\
           $f$ & $0$  & $0$ & $0$  \\
           $g$ & $f$  & $0$ & $e$  \\
        \end{tabular}.
        \end{center}\vskip3mm

   \begin{center}Table \ref{tab:d3-12}
   \begin{tabular}[t]{c|c}
         \hline\hline
        \bf{Leibniz $r$-matrices}  & \bf{Leibniz coalgebra induced by $r$-matrices}\\
        \hline
        $\makecell[c]{r=\l e \o e+\g e \o f+\g f \o e+\nu f \o f}$
        &$\left\{
            \begin{array}{ll}
             \d_r(e)=0\hspace{26mm}\\
             \d_r(f)=0\hspace{26mm}\\
             \d_r(g)=\l f\o e+\g f\o f\\
            \end{array}
            \right..$\\
        \hline
    \end{tabular}
    \end{center}\vskip3mm

   \item \label{tab:d3-11} For Leibniz algebra
       \begin{center}
        \begin{tabular}{r|rrr}
          $[,]$ & $e$  & $f$ & $g$ \\
          \hline
           $e$ & $0$  & $0$ & $0$ \\
           $f$ & $0$  & $0$ & $0$  \\
           $g$ & $e$  & $f$ & $0$  \\
        \end{tabular}.
        \end{center}\vskip3mm

   \begin{center}Table \ref{tab:d3-11}
   \begin{tabular}[t]{c|c}
         \hline\hline
        \bf{Leibniz $r$-matrices}  & \bf{Leibniz coalgebra induced by $r$-matrices}\\
        \hline
        $r=\l g\o g$
        &$\left\{
            \begin{array}{ll}
             \d_r(e)=\l e\o g-\l g\o e\\
             \d_r(f)=\l f\o g-\l g\o f\\
             \d_r(g)=0\hspace{24mm}\\
            \end{array}
            \right.$\\
        \hline
        $\makecell[c]{r=\l e \o e+\g e \o g+\g g \o e}$
        &$\left\{
            \begin{array}{ll}
             \d_r(e)=0\hspace{26mm}\\
             \d_r(f)=-\g e\o f+\g f\o e \\
             \d_r(g)=\l e\o e+\g e\o g\hspace{3mm}\\
            \end{array}
            \right.$\\
            \hline
        $\makecell[c]{r=\l e\o e+\g e\o f+\g f\o e+\nu f\o f}$
        &$\left\{
            \begin{array}{ll}
             \d_r(e)=0\hspace{56mm}\\
             \d_r(f)=0\hspace{56mm}\\
             \makecell[c]{\d_r(g)=\l e\o e+\g e\o f+\g f\o e\\+\nu f\o f}\\
            \end{array}
            \right.$\\
            \hline
        $\makecell[c]{r=\frac{\l \g}{\kappa} e \o e+\frac{\g \nu}{\kappa} (e \o f+f \o e)+\g (e\o g+\\g\o e)+\nu f \o f+\kappa (f\o g+g\o f),\\~\hbox{where~} \kappa\neq 0}$
        &$\left\{
            \begin{array}{ll}
             \d_r(e)=\kappa e\o f-\kappa f\o e\hspace{65mm}\\
             \d_r(f)=-\g e\o f+\g f\o e\hspace{63mm}\\
             \makecell[c]{\d_r(g)=\frac{\l \g}{\kappa} e\o e+\frac{\g \nu}{\kappa} (e \o f+f \o e)\\\qquad+\g e\o g+\nu f\o f+\kappa f\o g}\\
            \end{array}
            \right..$\\
            \hline
    \end{tabular}
    \end{center}

 \end{enumerate}
\end{thm}

 \section{Further research} In \cite{BGM}, we established a bialgebra theory of Rota-Baxter algebras by using dual representations. In this paper, we prove that some Leibniz bialgebras can produce Nijenhuis Leibniz algebras and Nijenhuis Leibniz coalgebras. Is there a suitable Nijenhuis Leibniz bialgebraic structure that precisely derives from Theorem \ref{thm:ln} and Theorem \ref{thm:cln}? Therefore we will try to investigate a bialgebraic structure of a Nijenhuis Leibniz algebra in future \cite{MSun}.

 \section{Acknowledgment} This work is supported by Natural Science Foundation of China (Nos. 12271089, 11871144) and Natural Science Foundation of Henan Province (Nos. 242300421389, 212300410365).

 \bigskip
 \noindent
{\bf Author Contributions} The authors made equal contributions to the paper.

 \bigskip

\noindent
{\bf Competing interests} The authors claim that there is no conflict of interests.

\end{document}